\documentclass{article}
\usepackage{amssymb}
\usepackage{amsthm}
\usepackage[curve,matrix,arrow]{xy}
\theoremstyle{plain}\newtheorem{Theorem}{Theorem}[section]
\theoremstyle{plain}
\theoremstyle{plain}\newtheorem{Corollary}[Theorem]{Corollary}
\theoremstyle{plain}\newtheorem{Lemma}[Theorem]{Lemma}
\theoremstyle{plain}
\theoremstyle{plain}
\theoremstyle{definition}
\theoremstyle{definition}
\theoremstyle{definition}
\theoremstyle{definition}\newtheorem{Remark}[Theorem]{Remark}
\theoremstyle{definition}
\theoremstyle{definition}\newtheorem{paragr}[Theorem]{}

  \def\OG{{\mathcal{O}G}}  
  \def\OH{{\mathcal{O}H}}  
  \def\OP{{\mathcal{O}P}}
  \def\OQ{{\mathcal{O}Q}}
  
\def\CF{{\mathcal{F}}}

\def\CO{{\mathcal{O}}}

\def\Aut{\mathrm{Aut}}           \def\tenk{\otimes_k}     
\def\Br{\mathrm{Br}}             \def\ten{\otimes}

           \def\tenkP{\otimes_{kP}}
\def\End{\mathrm{End}}

\def\foc{\mathfrak{foc}}

\def\Hom{\mathrm{Hom}}           

\def\hyp{\mathfrak{hyp}}
\def\ker{\mathrm{ker}}           
\def\Id{\mathrm{Id}}             \def\tenA{\otimes_A}
\def\Im{\mathrm{Im}}             \def\tenB{\otimes_B}
\def\Ind{\mathrm{Ind}}           
\def\Inn{\mathrm{Inn}}

\def\Irr{\mathrm{Irr}}           
           \def\tenO{\otimes_{\mathcal{O}}}
\def\mod{\mathrm{mod}}

\def\op{\mathrm{op}}
\def\Out{\mathrm{Out}}
\def\Pic{\mathrm{Pic}}

           \def\tenOP{\otimes_{\mathcal{O}P}}
         \def\tenOQ{\otimes_{\mathcal{O}Q}}

\title{On automorphisms and focal subgroups of blocks} 
\author{Markus Linckelmann} 
\date{}

\begin{document}

\maketitle

\begin{abstract}
Given a $p$-block $B$ of a finite group with defect group $P$ and
fusion system $\CF$ on $P$ we show that the rank of the group
$P/\foc(\CF)$ is invariant under stable equivalences of Morita type.
The main ingredients are the $*$-construction, due to Brou\'e and 
Puig, a theorem of Weiss on linear source modules, arguments
of Hertweck and Kimmerle applying Weiss' theorem to blocks, and
connections with integrable derivations in the Hochschild cohomology
of block algebras.
\end{abstract}

\section{Introduction}

Throughout this paper, $p$ is a prime, and $\CO$ is a complete discrete 
valuation ring with maximal ideal $J(\CO)=$ $\pi\CO$ for some
$\pi\in$ $\CO$, residue field $k=$ $\CO/J(\CO)$ of characteristic 
$p$, and field of fractions $K$ of characteristic zero. 
For any $\CO$-algebra $A$ which is free of finite rank as an $\CO$-module
and for any positive integer $r$ denote by $\Aut_r(A)$ the group of
$\CO$-algebra automorphisms $\alpha$ with the property that $\alpha$
induces the identity on $A/\pi^rA$, and denote by $\Out_r(A)$ the
image of $\Aut_r(A)$ in the outer automorphism group $\Out(A)=$
$\Aut(A)/\Inn(A)$ of $A$. 

Given a finite group $G$, a {\it block of $\OG$} is an indecomposable 
direct factor $B$ of $\OG$ as an algebra. Any such block $B$ determines 
a $p$-subgroup $P$ of $G$, called a {\it defect group of $B$}. A 
primitive idempotent $i$ in $B^P$ such that $\Br_P(i)\neq$ $0$ is called 
a {\it source idempotent}; the choice of a source idempotent determines 
a fusion system $\CF$ on $P$. We denote by $\foc(\CF)$ the $\CF$-focal 
subgroup of $P$; this is the subgroup of $P$ generated by all elements 
of the form $\varphi(u)u^{-1}$, where $u\in$ $P$ and $\varphi\in$ 
$\Hom_\CF(\langle u\rangle, P)$. Clearly $\foc(\CF)$ is a normal 
subgroup of $P$ containing the derived subgroup of $P$.

If $\CO$ is large enough, then the Brou\'e-Puig $*$-construction in 
\cite{BrPuloc} induces an action of the group 
$\Hom(P/\foc(\CF),\CO^\times)$ on the set $\Irr_K(B)$ of irreducible 
$K$-valued characters of $G$ associated with $B$, sending $\zeta\in$ 
$\Hom(P/\foc(\CF),\CO^\times)$ and $\chi\in$ $\Irr_K(B)$ to 
$\zeta*\chi\in$ $\Irr_K(B)$. The group $\Out(B)$ acts in the obvious way 
on $\Irr_K(B)$ by precomposing characters with automorphisms; that is, 
for $\alpha\in$ $\Aut(B)$ and $\chi\in$ $\Irr_K(B)$, viewed as a central 
function on $B$, the assignment $\chi^\alpha(x)=$ $\chi(\alpha(x))$ for 
all $x\in$ $G$ defines a character $\chi^\alpha \in$ $\Irr_K(B)$ which 
depends only on the image of $\alpha$ in $\Out(B)$.
See \S \ref{background} below for more details and references.

\begin{Theorem} \label{star-lift}
Let $G$ be a finite group. Let $B$ be a block algebra of $\OG$ with a 
nontrivial defect group $P$, source idempotent $i\in$ $B^P$
and associated fusion system $\CF$ on $P$. 
Suppose that $\CO$ contains a primitive $|G|$-th root of unity. 
Let $\tau_p$ be a primitive $p$-th root of unity in $\CO$ and let $m$ be 
the positive integer such that $\pi^m\CO=$ $(1-\tau_p)\CO$. Let $\mu$ be 
the subgroup of $\CO^\times$ generated by $\tau_p$.
There is a unique injective group homomorphism
$$\Phi : \Hom(P/\foc(\CF), \CO^\times) \to \Out_1(B)$$
such that for any $\zeta\in$ $\Hom(P/\foc(\CF), \CO^\times)$ the class 
$\Phi(\zeta)$ in $\Out_1(B)$ has a representative in $\Aut_1(B)$ 
which sends $ui$ to $\zeta(u)ui$ for all $u\in$ $P$. Moreover, $\Phi$
has the following properties.

\smallskip\noindent (i)
For any $\zeta\in$ $\Hom(P/\foc(\CF), \CO^\times)$ and any $\chi\in$ 
$\Irr_K(B)$ we have $\chi^{\Phi(\zeta)}=$ $\zeta * \chi$. 

\smallskip\noindent (ii)
If $\CO$ is finitely generated as a module over the ring of $p$-adic
integers, then the group homomorphism $\Phi$ restricts to an isomorphism
$\Hom(P/\foc(\CF), \mu) \cong \Out_m(B)$.
\end{Theorem}

\begin{Remark}
The group homomorphism $\Phi$ lifts the well-known action of
$\Hom(P/\foc(\CF),\CO^\times)$ on $\Irr_K(B)$ via the $*$-construction.
The existence of $\Phi$ as stated is a straightforward consequence of 
the hyperfocal subalgebra of a block. We will give a proof which does 
not require the hyperfocal subalgebra, based on some more general 
statements on automorphisms of source algebras in section 
\ref{autSection}. The point of statement (ii) is that the left side in 
the isomorphism depends on 
the fusion system of $B$ and the right side on the $\CO$-algebra 
structure of $B$. The extent of the connections between these two 
aspects of block theory remains mysterious. Numerous `local to global' 
conjectures predict that invariants of the fusion system of a block $B$ 
should essentially determine invariants of the $\CO$-algebra $B$, if 
not outright then up to finitely many possibilities. The `global to 
local' direction is perhaps even less understood: does the $\CO$-algebra
structure of a block algebra determine the key invariants on the local 
side, such as defect groups, fusion systems, and possibly 
K\"ulshammer-Puig  classes? 
\end{Remark}

\begin{Remark}
With the notation above, the group $P/\foc(\CF)$ has a topological 
interpretation: by \cite[Theorem 2.5]{BCGLO} this group is the 
abelianisation of the fundamental group of the $p$-completed nerve of a 
centric linking system of $\CF$.
\end{Remark}

The subgroup $\Hom(P/\foc(\CF),\mu)$ of $\Hom(P/\foc(\CF),\CO^\times)$
is isomorphic to the quotient of $P/\foc(\CF)$ by its Frattini 
subgroup. Since $P/\foc(\CF)$ is abelian, it follows that the rank 
of $\Hom(P/\foc(\CF),\mu)$ is equal to the rank of $P/\foc(\CF)$. 
Thus Theorem \ref{star-lift} has the following consequence.

\begin{Corollary} \label{star-lift-Cor1}
Suppose that $\CO$ is finitely generated as a module over the ring of 
$p$-adic integers. With the notation from \ref{star-lift}, the group 
$\Out_m(B)$ is a finite elementary abelian $p$-group of rank equal to 
the rank of the abelian $p$-group $P/\foc(\CF)$. In particular, if 
$P/\foc(\CF)$ is elementary abelian, then $\Out_m(B)\cong$ 
$P/\foc(\CF)$.
\end{Corollary}

Combining Theorem \ref{star-lift}  with invariance statements on the 
subgroups $\Out_m(B)$ from \cite{Linder} yields the following
statement.

\begin{Corollary} \label{star-lift-Cor2}
Suppose that $\CO$ is finitely generated as a module over the ring of 
$p$-adic integers. 
Let $G$, $G'$ be finite groups, and let $B$, $B'$ be block algebras
of $\OG$, $\OG'$ with nontrivial defect groups $P$, $P'$ and fusion 
systems $\CF$, $\CF'$ on $P$, $P'$, respectively. If there is a stable 
equivalence of Morita type between $B$ and $B'$, then the ranks of the 
abelian $p$-groups $P/\foc(\CF)$ and $P'/\foc(\CF')$ are equal.
\end{Corollary}

It remains an open question whether there is in fact an isomorphism
$P/\foc(\CF)\cong$ $P'/\foc(\CF')$ in the situation of this corollary.
If $P$ and $P'$ are elementary abelian, this follows trivially from
the above. In that case one can be slightly more precise, making use of
the following well-known facts. The Hochschild cohomology in positive 
degrees of a block algebra is invariant under stable equivalences of 
Morita type. In particular, a stable equivalence of Morita type between
two block algebras preserves the Krull dimensions of their Hochschild 
cohomology algebras over $k$, and these dimensions are equal to the rank 
of the defect groups. A stable equivalence of Morita type between two 
block algebras preserves also the order of the defect groups. 
A finite $p$-group which has the same order and rank as an elementary
abelian $p$-group is necessarily elementary abelian as well.

\begin{Corollary} \label{star-lift-Cor3}
Suppose that $\CO$ is finitely generated as a module over the ring of 
$p$-adic integers. 
With the notation of \ref{star-lift-Cor2}, if there is a stable 
equivalence of Morita type between $B$ and $B'$ and if one of $P$, 
$P'$ is elementary abelian, then there is an isomorphism $P\cong$ $P'$ 
which induces isomorphisms $\foc(\CF)\cong$ $\foc(\CF')$ and 
$P/\foc(\CF)\cong$ $P'/\foc(\CF')$.
\end{Corollary}

The main ingredients for the proof of Theorem \ref{star-lift} are 
results of Puig on source algebras of blocks, a theorem of Weiss 
\cite[Theorem 3]{Weiss}, and results from Hertweck and Kimmerle 
\cite{HeKi}. 

\medskip
Theorem \ref{star-lift} (ii) can be formulated in terms of integrable 
derivations, a concept due to Gerstenhaber \cite{Ger1}, adapted to 
unequal characteristic in \cite{Linder}. 
Let $A$ be an $\CO$-algebra such that $A$ is free
of finite rank as an $\CO$-module. Let $r$ be a positive integer and
let $\alpha\in$ $\Aut_r(A)$. Then $\alpha(a)=$ $a + \pi^r\mu(a)$ for 
all $a\in$ $A$ and some linear endomorphism $\mu$ of $A$. 
The endomorphism of $A/\pi^rA$ induced by $\mu$ is a derivation on 
$A/\pi^rA$. Any derivation on $A/\pi^rA$ which arises in this way is 
called {\it $A$-integrable}. The set of $A$-integrable derivations of 
$A/\pi^rA$ is an abelian group containing all inner derivations, hence 
determines a subgroup of $HH^1(A/\pi^rA)$, denoted $HH^1_A(A/\pi^rA)$.
Note that $\pi^r$ annihilates $HH^1(A/\pi^rA)$. Thus, if $p\in$
$\pi^rA$, then $HH^1_A(A/\pi^rA)$ is an elementary abelian quotient
of $\Out_r(A)$. See \cite[\S 3]{Linder} for more details.

\begin{Theorem} \label{der-foc}
Let $G$ be a finite group. Let $B$ a block of $\OG$ with a nontrivial 
defect group $P$ and a fusion system $\CF$ on $P$. Suppose that $\CO$ 
contains a primitive $|G|$-th root of unity and that $\CO$ is finitely 
generated as a module over the ring of $p$-adic integers. 
Denote by $\tau_p$ a primitive $p$-th root of unity in $\CO$, and let 
$m$ be the positive integer such that $\pi^m\CO=$ $(1-\tau_p)\CO$.
We have a canonical group isomorphism 
$$\Out_m(B)\cong HH^1_B(B/\pi^mB)\ .$$ 
In particular, $HH^1_B(B/\pi^mB)$ is a finite elementary abelian 
$p$-group of rank equal to the rank of $P/\foc(\CF)$.
\end{Theorem}

\begin{Remark}
The group homomorphism $\Phi$ in Theorem \ref{star-lift} depends on
the choice of $P$ and $i$, but it is easy to describe the impact on 
$\Phi$ for a different choice. By \cite[Theorem 1.2]{Puigpoint}, if $P'$ 
is another defect group of $B$ and $i'\in$ $B^{P'}$ a source idempotent, 
then there is $x\in$ $G$ such that ${^xP}=$ $P'$ and such that ${^xi}$ 
belongs to the same local point of $P'$ on $B$ as $i'$. Thus we may 
assume that $i'=$ ${^xi}$. Conjugation by $x$ sends the fusion system 
$\CF$ on $P$ determined by $i$ to the fusion system $\CF'$ on $P'$ 
determined by the choice of $i'$, hence induces group isomorphisms 
$\foc(\CF)\cong$ $\foc(\CF')$ and $P/\foc(\CF)\cong$ $P'/\foc(\CF')$. 
This in turn induces a group isomorphism 
$\nu : \Hom(P'/\foc(\CF'),\CO^\times)\cong$ 
$\Hom(P/\foc(\CF), \CO^\times)$. Precomposing $\Phi$ with $\nu$ yields 
the group homomorphism $\Phi'$ as in Theorem \ref{star-lift} for $P'$, 
$i'$, $\CF'$ instead of $P$, $i$, $\CF$, respectively. Indeed, if an 
automorphism $\beta$ of $B$ in $\Aut_1(B)$ sends $ui$ to $\zeta(u)ui$, 
then conjugating $\beta$ by $x$ yields an automorphism $\beta'$ in 
$\Aut_1(B)$ belonging to the same class as $\beta$ in $\Out_1(B)$, and 
by construction, $\beta'$ sends $u'i'$ to $\zeta'(u')u'i'$, where 
$u'\in$ $P'$ and $\zeta'$ corresponds to $\zeta$ via $\nu$. 
\end{Remark}

\section{Background material} \label{background}

\begin{paragr} \label{fus-para}
The terminology on fusion systems required in this paper can be found 
in \cite[Part I]{AKO}. By a fusion system we always mean a saturated 
fusion system in the sense of \cite[I.2.2]{AKO}. For a broader treatment 
on fusion systems, see \cite{CravenBook}. Given a fusion system $\CF$ on 
a finite $p$-group $P$, the {\it focal subgroup of $\CF$ in $P$} is the 
subgroup, denoted $\foc(\CF)$, generated by all elements of the form 
$u^{-1}\varphi(u)$, where $u$ is an element of a subgroup $Q$ of $P$ and 
where $\varphi\in$ $\Aut_\CF(Q)$. The focal subgroup $\foc(\CF)$ is
normal in $P$ and contains the derived subgroup $P'$ of $P$; thus
$P/\foc(\CF)$ is abelian. The focal subgroup contains the {\it 
hyperfocal subgroup} $\hyp(\CF)$ generated by all elements 
$u^{-1}\varphi(u)u^{-1}$ as above with the additional condition that 
$\varphi$ has $p'$-order. We have $\foc(\CF)=$ $\hyp(\CF) P'$. Both 
$\foc(\CF)$ and $\hyp(\CF)$ are not only normal in $P$ but in fact 
stable under $\Aut_\CF(P)$. A subgroup $Q$ of $P$ is called 
{\it $\CF$-centric} if for any $\varphi\in$ $\Hom_\CF(Q,P)$ we have
$C_P(\varphi(Q))=$ $Z(\varphi(Q))$. As a consequence of Alperin's fusion 
theorem, $\foc(\CF)$ is generated by all elements of the form 
$u^{-1}\varphi(u)$, where $u$ is an element of an $\CF$-centric subgroup 
$Q$ of $P$ and where $\varphi\in$ $\Aut_\CF(Q)$. See \cite[I, \S 7]{AKO} 
for more details on focal and hyperfocal subgroups. 
\end{paragr}

\begin{paragr} \label{source-para}
We describe in this paragraph the definition and properties of source 
algebras which we will need in this paper. For introductions to some of 
the required block theoretic background material, see for instance 
\cite{Thev} and \cite[Part IV]{AKO}. 

Given a finite group $G$, a block of $\OG$ is an indecomposable direct 
factor $B$ of $\OG$ as an $\CO$-algebra. The unit element $b=$ $1_B$ of 
$B$ is then a primitive idempotent in $Z(\OG)$, called the {\it block 
idempotent of $B$}. For $P$ a $p$-subgroup of
$G$ we denote by $\Br_P : (\OG)^P \to$ $kC_G(P)$ the {\it Brauer
homomorphism} induced by the map sending $x\in$ $C_G(P)$ to its 
image in $kC_G(P)$ and sending $x\in$ $G\setminus C_G(P)$ to zero.
This is a surjective algebra homomorphism. Thus $\Br_P(b)$ is
either zero, or an idempotent in $kC_G(P)^{N_G(P)}$. 
If $P$ is maximal subject to the condition $\Br_P(b)\neq$ $0$,
then $P$ is called a {\it defect group} of $B$. The defect groups
of $B$ are conjugate in $G$. 

The condition $\Br_P(b)\neq$ $0$ implies that there is a primitive 
idempotent $i$ in $B^P$ such that $\Br_P(i)\neq$ $0$. The idempotent $i$ 
is then called a {\it source idempotent of} $B$ and the algebra 
$A=$ $iBi=$ $i\OG i$ is then called a {\it source algebra of} $B$. We 
view $A$ as an {\it interior $P$-algebra}; that is, we keep track of the 
image $iP$ of $P$ in $A$ via the group homomorphism $P\to$ 
$A^\times$ sending $u\in$ $P$ to $ui=iu=iui$. This group 
homomorphism is injective and induces an injective algebra homomorphism
$\OP$ which has a complement as an $\OP$-$\OP$-bimodule, because $A$ 
is projective as a left or right $\OP$-module. As an 
$\CO(P\times P)$-module, $A$ is a direct summand of $\OG$, and hence
$i\OG i$ is a permutation $\CO(P\times P)$-module. The isomorphism class 
of $A$ as an interior $P$-algebra is unique up to conjugation by 
elements in $N_G(P)$. By \cite[3.6]{Puigpoint} the source algebra $A$ 
and the block algebra $B$ are Morita equivalent via the bimodules $Bi=$ 
$\OG i$ and $iB=$ $i\OG$. This Morita equivalence induces an isomorphism 
$\Out_1(A)\cong$ $\Out_1(B)$; see Lemma \ref{AextendB} below for a more
precise statement. The strategy to prove Theorem \ref{star-lift} is 
to construct a group homomorphism $\Hom(P/\foc(\CF))\to$ $\Out_1(A)$ and 
then show that its composition with the isomorphism $\Out_1(A)\cong$ 
$\Out_1(B)$ satisfies the conclusions of Theorem \ref{star-lift}.

It follows from work of Alperin and Brou\'e \cite{AlBr} that $B$ 
determines a fusion system on any defect group $P$, uniquely up to 
conjugation. By work of Puig \cite{Puigloc}, every choice of a source 
algebra $A$ determines a fusion system $\CF$ on $P$. 
More precisely, the fusion system $\CF$ is determined by the 
$\OP$-$\OP$-bimodule structure of $A$: every indecomposable direct
summand of $A$ as an $\OP$-$\OP$-bimodule is isomorphic to
$\OP_\varphi \tenOQ \OP$ for some $\varphi\in$ $\Hom_\CF(Q,P)$, and the
morphisms in $\CF$ which arise in this way generate $\CF$. Here 
$\OP_\varphi$ is the $\OP$-$\OQ$-bimodule which is equal to $\OP$ as
a left $\OP$-module, and on which $u\in$ $Q$ acts on the right by
multiplication with $\varphi(u)$.  See \cite[\S 7]{Lisplendid} for an 
expository account of this material. 
Fusion systems on a defect group $P$ of $B$ obtained from different 
choices of source idempotents are $N_G(P)$-conjugate.

By \cite[Theorem 1.8]{Puighyp}, the source algebra $A$ has, up to 
conjugation by $P$-stable invertible elements in $A$, a unique unitary 
$P$-stable subalgebra $D$, called {\it hyperfocal subalgebra of 
$i\OG i$}, such that $D\cap Pi=$ $\hyp(\CF)i$ and such that $A =$ 
$\oplus_{u} Du$, with $u$ running over a set of representatives in $P$ 
of $P/\hyp(\CF)$.
\end{paragr}

\begin{paragr} \label{autom-para}
Let $A$ be an $\CO$-algebra which is free of finite rank as an 
$\CO$-module. In what follows the use of automorphisms as subscripts 
to modules is as in \cite{Listable}. That is, if $\alpha\in$ $\Aut(A)$ 
we denote for any $A$-module $U$ by ${_\alpha{U}}$ the $A$-module which 
is equal to $U$ as an $\CO$-module, with $a\in$ $A$ acting as 
$\alpha(a)$ on $U$. If $\alpha$ is inner, then ${_\alpha{U}}\cong$
$U$. We use the analogous notation for right modules and bimodules. If 
$U$ and $V$ are $A$-$A$-bimodules and $\alpha\in$ $\Aut(A)$, then we 
have an obvious isomorphism of $A$-$A$-bimodules
$(U_{\alpha})\tenA V\cong$ $U\tenA ({_{\alpha^{-1}}{V}})$.
We need the following standard fact (we sketch a proof for the
convenience of the reader).

\begin{Lemma} \label{autom-extend}
Let $A$ be an $\CO$-algebra and $B$ a subalgebra of $A$. Let
$\alpha\in$ $\Aut(A)$ and let $\beta : B\to$ $A$ be an $\CO$-algebra
homomorphism. The following are equivalent.

\smallskip\noindent (i)
There is an automorphism $\alpha'$ of $A$ which extends the map $\beta$ 
such that $\alpha$ and $\alpha'$ have the same image in $\Out(A)$.

\smallskip\noindent (ii)
There is an isomorphism of $A$-$B$-bimodules ${A_\beta }\cong$ 
${A_\alpha }$.

\smallskip\noindent (iii)
There is an isomorphism of $B$-$A$-bimodules ${_\beta A}\cong$ 
${_\alpha A}$.
\end{Lemma}

\begin{proof}
Clearly (i) implies (ii) and (iii). Suppose that (ii) holds.
An $A$-$B$-bimodule isomorphism
$\Phi :{A_\beta }\cong$ ${A_\alpha}$ is in particular a left
$A$-module automorphism of $A$, hence induced by right multiplication
with an element $c\in$ $A^\times$. The fact that $\Phi$ is also
a homomorphism of right $B$-modules implies that $\beta(b)c=$ 
$c\alpha(b)$ for all $b\in$ $B$. Thus $\alpha'$ defined by
$\alpha'(a)= c\alpha(a)c^{-1}$ for all $a\in$ $A$ defines an
automorphism of $A$ which extends $\beta$ and whose class in 
$\Out(A)$ coincides with that of $\alpha$. Thus (ii) implies (i).
A similar argument shows that (iii) implies (i).
\end{proof}

A frequently used special case of Lemma \ref{autom-extend} (with $B=A$ 
and $\beta=$ $\Id$) is that $A\cong$ ${A_\alpha}$ as $A$-$A$-bimodules 
if and only if $\alpha$ is inner. Note that besides being an algebra
automorphism, $\alpha$ is also an isomorphism of $A$-$A$-bimodules 
${A_{\alpha^{-1}}}\cong$ ${_{\alpha}A}$. Any $A$-$A$-bimodule of the form
${A_\alpha}$ for some $\alpha\in$ $\Aut(A)$ induces a Morita 
equivalence on $A$, with inverse equivalence induced by 
${A_{\alpha^{-1}}}$. An $A$-$A$-bimodule $M$ which induces a Morita 
equivalence on $\mod(A)$ is of the form ${A_\alpha}$ for some
$\alpha\in$ $\Aut(M)$ if and only if $M\cong$ $A$ as left 
$A$-modules, which is also equivalent to $M\cong$ $A$ as right
$A$-modules. This embeds $\Out(A)$ as a subgroup of $\Pic(A)$.
This embedding identifies $\Out_r(A)$ with the kernel of the canonical 
homomorphism of Picard groups $\Pic(A)\to$ $\Pic(A/\pi^rA)$, where $r$ 
is a positive integer. See e. g. \cite[\S 55 A]{CR2} for more details.
\end{paragr}

\begin{paragr} \label{stable-para}
Let $A$ and $B$ be $\CO$-algebras which are free of finite ranks as
$\CO$-modules. Let $M$ be an $A$-$B$-bimodule such that $M$ is finitely 
generated projective as a left $A$-module and as a right $B$-module. 
Let $N$ be a $B$-$A$-bimodule which is finitely generated as a left
$B$-module and as a right $A$-module. Following Brou\'e \cite{BroueEq}
we say that {\it $M$ and $N$ induce a stable equivalence of Morita type
between $A$ and $B$} if we have isomorphisms $M\tenB N\cong$ $B\oplus Y$
and $N\tenA M\cong$ $A\oplus X$ as $B\tenO B^\op$-modules and
$A\tenO A^\op$-modules, respectively, such that $Y$ is a projective
$B\tenO B^\op$-module and $X$ is a projective $A\tenO A^\op$-module.
\end{paragr}

\begin{Theorem}[{\cite[Theorem 4.2]{Listable}, 
\cite[Lemma 5.2]{Linder}}] \label{outmstable} 
Let $A$, $B$ be $\CO$-algebras which are free of finite rank as 
$\CO$-modules, such that the $k$-algebras  $k\tenO A$ and $k\tenO B$ are 
indecomposable nonsimple selfinjective with separable semisimple 
quotients. Let $r$ be a positive integer. Suppose that the canonical 
maps $Z(A)\to$ $Z(A/\pi^rA)$ and $Z(B)\to$ $Z(B/\pi^rB)$ are surjective. 
Let $M$ be an $A$-$B$-bimodule and $N$ a $B$-$A$-bimodule inducing a 
stable equivalence of Morita type between $A$ and $B$. For any 
$\alpha\in$ $\Aut_r(A)$ there is $\beta\in$ $\Aut_r(B)$ such that 
${_{\alpha^{-1}}{M}}\cong$ $M_\beta$ as $A$-$B$-bimodules, and the 
correspondence $\alpha\mapsto$ $\beta$ induces a group isomorphism 
$\Out_r(A)\cong$ $\Out_r(B)$.
\end{Theorem}

If $M$ and $N$ induce a Morita equivalence, then the hypothesis on
$k\tenO A$ and $k\tenO B$ being selfinjective with separable semisimple
quotients is not needed; see the \cite[Remark 5.4]{Linder} for the 
necessary adjustments. 
The proof is a variation of \cite[Theorem 4.2]{Listable}; details
can be found in \cite[Lemma 5.2]{Linder}. The surjectivity hypothesis
for the map $Z(A)\to$ $Z(A/\pi^rA)$ ensures that two automorphisms
in $\Aut_r(A)$ represent the same class in $\Out(A)$ if and only if
they differ by conjugation with an element in $1+\pi^rA$; equivalently,
we have $\Inn_r(A)=$ $\Inn(A)\cap \Aut_1(A)$, where $\Inn_r(A)$ is the
subgroup of $\Inn(A)$ consisting of automorphisms given by conjugation
with elements in $1+\pi^rA$; this follows from \cite[3.2]{Linder}. 
We will use this fact without further reference.

Given two finite groups $G$, $H$, we consider any $\OG$-$\OH$-bimodule 
$M$ as an $\CO(G\times H)$-module via $(x,y)\cdot m=$ $xmy^{-1}$, where
$x\in$ $G$, $y\in$ $H$, $m\in$ $M$. The easy proof of the following
well-known Lemma is left to the reader.

\begin{Lemma} \label{zetaeta}
Let $G$ be a finite group and $\zeta : G\to$ $\CO^\times$ a group
homomorphism. Denote by $\CO_\zeta$ the $\OG$-module which is equal
to $\CO$ as an $\CO$-module and on which any $x\in$ $G$ acts as
multiplication by $\zeta(x)$. Set $\Delta G=$ $\{(x,x)\ |\ x\in G\}$
and consider $\CO_\zeta$ as a module over $\CO\Delta G$ via the
canonical group isomorphism $\Delta G\cong$ $G$.

\smallskip\noindent (i)
The $\CO$-linear endomorphism $\eta$ of $\OG$ defined by $\eta(x)=$
$\zeta(x)x$ for all $x\in$ $G$ is an $\CO$-algebra automorphism of
$\OG$, and the map $\zeta\mapsto\eta$ induces an injective group
homomorphism $\Hom(G,\CO^\times)\to$ $\Out(\OG)$. 

\smallskip\noindent (ii)
There is an isomorphism of $\CO(G\times G)$-modules
${_\eta{\OG}}\cong$ $\Ind^{G\times G}_{\Delta G}(\CO_\zeta)$
which sends $x\in$ $G$ to $\zeta(x^{-1})(x,1)\ten 1$. 
\end{Lemma}

\section{On automorphisms of source algebras}
\label{autSection}

\begin{Theorem} \label{PfocAut}
Let $G$ be a finite group and $B$ a block of $\OG$ with a nontrivial 
defect group. Let $i$ be a source idempotent in $B^P$ and $A=$ $iBi$ the 
corresponding source algebra of $B$ with associated fusion system $\CF$ 
on $P$. Assume that $\CO$ contains a primitive $|G|$-th root of unity. 
Identify $\OP$ to its image in $A$. Let $\zeta\in$ $\Hom(P,\CO^\times)$. 

There is $\alpha\in$ $\Aut_1(A)$ satisfying $\alpha(u)=$ $\zeta(u)u$ for
all $u\in$ $P$ if and only if $\foc(\CF)\leq$ $\ker(\zeta)$. In that
case the class of $\alpha$ in $\Out_1(A)$ is uniquely determined by
$\zeta$, and the correspondence $\zeta\mapsto\alpha$ induces an
injective group homomorphism $\Psi : \Hom(P/\foc(\CF),\CO^\times)\to$
$\Out_1(A)$.
\end{Theorem}

We denote by $\Aut_P(A)$ the group of $\CO$-algebra 
automorphisms of $A$ which preserve the image of $P$ in $A$ 
elementwise; that is, $\Aut_P(A)$ is the group of automorphisms of $A$ 
as an interior $P$-algebra. By a result of Puig 
\cite[14.9]{Puigmodules}, the group $\Aut_P(A)$ is canonically 
isomorphic to a subgroup of the $p'$-group $\Hom(E,k^\times)$, where 
$E$ is the inertial quotient of $B$ associated with $A$. 

Whenever convenient, we identify the 
elements in $\Hom(P/\foc(\CF),\CO^\times)$ with the subgroup of all 
$\zeta\in$ $\Hom(P,\CO^\times)$ satisfying $\foc(\CF)\leq$ $\ker(\zeta)$.
Note that if $\zeta\in$ $\Hom(P,\CO^\times)$ and if $\eta$ is the 
automorphism of $\OP$ defined by $\eta(u)=$ $\zeta(u)u$ for all $u\in$ 
$P$, then $\eta\in$ $\Aut_1(\OP)$ because the image in $k$ of any 
$p$-power root of unity in $\CO$ is equal to $1_k$. 

The fastest way to show the existence of $\alpha$ as stated uses the 
hyperfocal subalgebra. We state this as a separate Lemma, since we will 
give at the end of this section a second proof of this fact which does 
not require the hyperfocal subalgebra.

\begin{Lemma} \label{zeta-extend1}
Let $\zeta\in$ $\Hom(P, \CO^\times)$ such that $\foc(\CF)\leq$
$\ker(\zeta)$. Then there is $\alpha\in$ $\Aut_1(A)$ such that
$\alpha(u)=$ $\zeta(u)u$ for all $u\in$ $P$.
\end{Lemma}

\begin{proof}
Let $D$ be a hyperfocal subalgebra of the source algebra $A$ of $B$.
That is, $D$ is a $P$-stable subalgebra of $A$ such that $D\cap Pi=$ 
$\hyp(\CF)$ and $A=$ $\oplus_u Du$, where $u$ runs over a set of 
representatives of $P/\hyp(\CF)$ in $P$. For $d\in$ $D$ and $u\in$ $P$ 
define $\alpha(du)=$ $\zeta(u)du$. Since $\foc(\CF)\leq$ $\ker(\zeta)$, 
this is well-defined, and extends linearly to $A$. A trivial 
verification shows that this is an $\CO$-algebra automorphism of $A$ 
which acts as the identity on $D$. The image of the $p$-power root of 
unity $\zeta(u)$ in $k$ is $1$, and hence $\alpha\in$ $\Aut_1(A)$. 
\end{proof}

\begin{Lemma} \label{zeta-extend2}
Let $\zeta\in$ $\Hom(P, \CO^\times)$ such that there exists
$\alpha\in$ $\Aut(A)$ satisfying $\alpha(u)=$ $\zeta(u)u$ for all 
$u\in$ $P$. Then $\foc(\CF)\leq$ $\ker(\zeta)$. 
\end{Lemma}

\begin{proof}
Let $Q$ be an $\CF$-centric subgroup of $P$ and $\varphi\in$
$\Aut(Q)$. We need to show that $u^{-1}\varphi(u)\in$ $\ker(\zeta)$
for any $u\in$ $Q$. Since $Q$ is $\CF$-centric, it follows from
\cite[(41.1)]{Thev} that $Q$ has a unique local point $\delta$ on $A$.
Let $\alpha\in$ $\Aut(A)$ such that $\alpha(u)=$ $\zeta(u)u$ for
all $u\in$ $P$. Note that $u$ and $\alpha(u)$ act in the same way
by conjugation on $A$ because they differ by a scalar. Thus
$\alpha$ induces an automorphism on $A^Q$ which preserves the ideals
of relative traces $A^Q_R$, where $R$ is a subgroup of $Q$. It follows
that $\alpha$ permutes the points of $Q$ on $A$ preserving the property
of being local. The uniqueness of $\delta$ implies 
that $\alpha(\delta)=$
$\delta$. Let $j\in$ $\delta$. Thus $\alpha(j) = $ $d^{-1}jd$ for some
$d\in$ $(A^Q)^\times$. After replacing $\alpha$, if necessary, we may 
assume that $\alpha$ fixes $j$ and still satisfies $\alpha(u)=$ 
$\zeta(u)u$ for all $u\in$ $Q$. By \cite[2.12, 3.1]{Puigloc} 
there exists $c\in$ $A^\times$ such that ${^c(uj)}=$ $\varphi(u)j$ for
all $u\in$ $Q$. Applying $\alpha$ to this equation yields
$${^{\alpha(c)}\zeta(u)uj} = \zeta(\varphi(u))\varphi(u)j$$
for all $u\in$ $Q$. Conjugating by $c^{-1}$ and multiplying by 
$\zeta(u^{-1})$ yields
$${^{c^{-1}\alpha(c)} uj} = \zeta(u^{-1}\varphi(u)) u j$$
for all $u\in$ $Q$. Set $\kappa(u)=\zeta(u^{-1}\varphi(u))$. 
This defines a group homomorphism $\kappa : Q\to$ $\CO^\times$.
We need to show that $\kappa$ is the trivial group homomorphism.
Since $c^{-1}\alpha(c)$ centralises $j$, it follows that the
element $w=$ $c^{-1}\alpha(c)j$ belongs to $(jAj)^\times$. Moreover,
by the above, conjugation by $w$ on $jAj$ induces an inner automorphism
$\sigma$ of $jAj$ whose restriction to $\OQ$ (identified to its image
$\OQ j$ in $jAj$) is the automorphism $\theta$ of $\OQ$ defined by 
$\theta(u)=$ $\kappa(u)u$ for all $u\in$ $Q$.  Since $\sigma$ is inner, 
we have $jAj_\sigma\cong$ $jAj$ as $jAj$-$jAj$-bimodules. Thus we have
$jAj_\theta\cong$ $jAj$ as $\OQ$-$\OQ$-bimodules. In particular,
$jAj_\theta$ is a permutation $\CO(Q\times Q)$-module. Since
$\Br_Q(j)\neq$ $0$, it follows that $\OQ$ is a direct summand of
$jAj$ as an $\CO(Q\times Q)$-module. Thus $\OQ_\theta$ is
a direct summand of $jAj_\theta$ as an $\CO(Q\times Q)$-module.
But $\OQ_\theta$ is a permutation $\CO(Q\times Q)$-module if and only
if $\kappa=1$. The result follows.
\end{proof} 

The two lemmas \ref{zeta-extend1} and \ref{zeta-extend2} prove the
first statement of Theorem \ref{PfocAut}. The following lemmas collect
the technicalities needed for proving the remaining statements of
Theorem \ref{PfocAut}.

\begin{Lemma} \label{AalphaA}
Let $\alpha\in$ $\Aut(A)$. Then $A_\alpha$ is isomorphic to a direct
summand of $A_\alpha\tenOP A$ as an $A$-$A$-bimodule.
\end{Lemma}

\begin{proof}
By \cite[4.2]{Liperm}, $A$ is isomorphic to a direct summand of 
$A\tenOP A$. Tensoring by $A_\alpha\tenA -$ yields the result.
\end{proof}

\begin{Lemma}\label{Aut1Inn}
Let $\alpha\in$ $\Aut_1(A)\cdot\Inn(A)$ such that $A_\alpha$ is
isomorphic to a direct summand of $A\tenOP A$ as an $A$-$A$-bimodule. 
Then $\alpha\in$ $\Inn(A)$.
\end{Lemma}

\begin{proof}
Tensoring $A_\alpha$ and $A\tenOP A$ by $Bi\tenA - \tenA iB$ implies that 
$Bi_\alpha\tenA iB$ is isomorphic to a direct summand of $Bi\tenOP iB$. 
This shows that $Bi_\alpha\tenA iB$ is a $p$-permutation 
$\CO(G\times G)$-module. Since $\alpha\in$ $\Aut_1(A)\cdot\Inn(A)$, it 
follows that $\alpha$ induces an inner automorphism on $k\tenO A$.
Thus $k\tenO Bi_\alpha\tenA iB\cong$ $k\tenO Bi\tenA iB$.
The fact that $p$-permutation $k(G\times G)$-modules lift
uniquely, up to isomorphism, to $\CO(G\times G)$-modules implies that
$Bi_\alpha\tenA iB\cong$ $Bi\tenA iB$ as $\CO(G\times G)$-modules.
Mutliplying both modules on the left and on the right by $i$ implies 
that $A_\alpha\cong$ $A$ as $A$-$A$-bimodules, and hence
$\alpha$ is inner.
\end{proof}

\begin{Lemma} \label{Aut1capP}
We have $\Aut_P(A)\cap \Aut_1(A)\cdot \Inn(A) \leq$ $\Inn(A)$.
\end{Lemma}

\begin{proof}
Let $\alpha\in$ $\Aut_P(A)\cap \Aut_1(A)\cdot \Inn(A)$.
By \ref{AalphaA}, $A_\alpha$ is a isomorphic to a direct summand of
$A_\alpha\tenOP A\cong$ $A\tenOP A$, where last isomorphism uses the 
fact that $\alpha$ fixes the image of $P$ in $A$. Thus $\alpha$ is 
inner by \ref{Aut1Inn}.
\end{proof}

The following lemma shows that there is a well-defined group
homomorphism $\Psi : \Hom(P/\foc(\CF))$ $\to$ $\Out_1(A)$ as stated
in Theorem \ref{PfocAut}.

\begin{Lemma} \label{AutOPequal}
Let $\alpha$, $\alpha'\in$ $\Aut_1(A)\cdot\Inn(A)$ such that 
$\alpha(u)=$ $\alpha'(u)$ for all $u\in$ $P$. Then $\alpha$ and 
$\alpha'$ have the same image in $\Out(A)$.
\end{Lemma}

\begin{proof}
The automorphism $\alpha^{-1}\circ\alpha'$ belongs to 
$\Aut_1(A)\cdot\Inn(A)$ and acts as identity on $P$. Thus this 
automorphism is inner by \ref{Aut1capP}. The result follows.
\end{proof}

\begin{Remark}
The assumption in the previous lemma that both $\alpha$, $\alpha'$ 
belong to $\Aut_1(A)\cdot\Inn(A)$ is necessary. For instance, if $p$ is 
odd, $P$ is cyclic of order $p$, and $E\leq$ $\Aut(P)$ is the subgroup 
of order $2$, then $\Id_\OP$ extends to the automorphism $\beta$ in 
$\Aut_P(\CO P\rtimes E)$ sending the nontrivial element $t$ of $E$ to 
$-t$. Clearly $\beta$ does not induce an inner automorphism on 
$kP\rtimes E$.
\end{Remark}

The next lemma shows that the group homomorphism $\Psi$
in the statement of Theorem \ref{PfocAut} is injective; this
completes the proof of Theorem \ref{PfocAut}.

\begin{Lemma} \label{zeta-extend-unique}
Let $\zeta\in$ $\Hom(P,\CO^\times)$. 
Suppose that there is $\alpha\in$ $\Aut(A)$ such that $\alpha(u)=$
$\zeta(u)u$ for all $u\in$ $P$. If $\alpha$ is inner, then $\zeta=$ $1$.
\end{Lemma}

\begin{proof}
Denote by $\eta$ the automorphism of $\OP$ defined by $\eta(u)=$
$\zeta(u)u$ for all $u\in$ $P$. Suppose that $\alpha$ is inner. Then 
${A_\alpha}\cong$ $A$ as $A$-$A$-bimodules. Since $\alpha$ extends 
$\eta$, it follows that $A_\eta\cong$ $A$ as $\OP$-$\OP$-bimodules. 
In particular, ${A_\eta}$ is a permutation $\CO(P\times P)$-module. 
Since $\OP$ is a direct summand of $A$ as an $\CO(P\times P)$-module, 
it follows that $A_\eta$ has $\OP_\eta$ as an indecomposable 
direct summand. This is a trivial source $\CO(P\times P)$-module if and 
only if $\zeta=1$, whence the result.
\end{proof}

The connection with automorphisms of $B$ is described in the following
observation, combining some of the standard facts on automorphisms
mentioned previously.

\begin{Lemma} \label{AextendB}
Every $\alpha\in$ $\Aut_1(A)$ extends to an automorphism $\beta\in$ 
$\Aut_1(B)$, and the correspondence $\alpha\mapsto\beta$ induces a
group isomorphism $\Out_1(A)\cong$ $\Out_1(B)$.
\end{Lemma}

\begin{proof}
The algebras $A$ and $B$ are Morita equivalent via the bimodules $Bi$ 
and $iB$. Let $\alpha\in$ $\Aut_1(A)$. By Theorem \ref{outmstable} there 
is $\beta\in$ $\Aut_1(B)$ such that ${_{\beta^{-1}}Bi}\cong$ $Bi_\alpha$ 
as $B$-$A$-bimodules, and the correspondence $\alpha\mapsto\beta$ 
induces a group isomorphism $\Out_1(A)\cong$ $\Out_1(B)$. We need to 
show that $\beta$ can be chosen in such a way that it extends $\alpha$. 
Since $\beta$ induces the identity on $B/\pi B$, it follows from 
standard lifting idempotent theorems that $i$ and $\beta(i)$ are 
conjugate in $B^\times$ via an element in $1+\pi B$. Thus, after 
replacing $\beta$ in its class if necessary, we may assume that 
$\beta(i)=i$. It follows that $\beta$ restricts to an automorphism 
$\alpha'$ in $\Aut_1(A)$ representing the same class as $\alpha$ in 
$\Out_1(A)$. Thus $\alpha$ is equal to $\gamma\circ\alpha'$ for
some inner automorphism $\gamma$ of $A$ given by conjugation with an 
element $c\in$ $i+\pi A\subseteq$ $A^\times$. Therefore, denoting
by $\delta$ the inner automorphism of $B$ given by conjugation 
with $1-i+c\in$ $1+\pi B$, it follows that $\delta\circ\beta$ 
extends $\alpha$. 
\end{proof} 

In the remainder of this section, we give a proof of Lemma 
\ref{zeta-extend1} which does not require the hyperfocal subalgebra. We 
start by showing that certain automorphisms in $\Aut_1(A)\cdot\Inn(A)$
can be conjugated into $\Aut_1(A)$ via a $P$-stable invertible
element, and deduce that $\Im(\Psi)$ act trivially on pointed groups 
on $A$.

\begin{Lemma} \label{zeta-points}
Let $\zeta\in$ $\Hom(P, \CO^\times)$ such that there exists
$\alpha\in$ $\Aut_1(A)\Inn(A)$ satisfying $\alpha(u)=$ $\zeta(u)u$ for 
all $u\in$ $P$. 

\smallskip\noindent (i)
There is $c\in$ $(A^P)^\times$ such that the automorphism
$\alpha'$ defined by $\alpha'(a)=$ $c^{-1}\alpha(a)c$ for all $a\in$
$A$ satisfies $\alpha'\in$ $\Aut_1(A)$.

\smallskip\noindent (ii)
We have $\alpha'(u)=$ $\zeta(u)u$ for all $u\in$ $P$, and the classes 
of $\alpha'$ and of $\alpha$ in $\Out(A)$ are equal.
 
\smallskip\noindent (iii)
For any pointed group $Q_\epsilon$ on $A$ we have 
$\alpha(\epsilon)=$ $\epsilon$.
\end{Lemma}

\begin{proof}
By the assumptions on $\alpha$, the automorphism $\bar\alpha$
induced by $\alpha$ on $\bar A=$ $k\tenO A$ is inner and fixes
the image of $P$ in $\bar A$. Thus $\bar\alpha$ is induced by
conjugation with an invertible element $\bar c\in$ $(\bar A^P)^\times$.
The map $A^P\to$ $\bar A^P$ is surjective, hence so is the induced map
$(A^P)^\times\to$ $(\bar A^P)^\times$. Choose an inverse image
$c\in$ $(A^P)^\times$ of $\bar c$. Then $\alpha'$ as defined satisfies 
(i). Since conjugation by $c$ fixes the image of $P$ in $A$, statement
(ii) follows from (i). Conjugation by $c$ 
fixes any subgroup $Q$ of $P$ and hence any point $\epsilon$ of $Q$ on 
$A$. Thus, in order to prove (iii), we may replace $\alpha$ by 
$\alpha'$; that is, we may assume that $\alpha\in$ $\Aut_1(A)$. The 
hypotheses on $\alpha$ imply that $\alpha(\OQ)=$ $\OQ$. Since $A^Q$ is 
the centraliser in $A$ of $\OQ$ it follows that $\alpha$ restricts to an 
automorphism of $A^Q$. The canonical map $A^Q\to$ $(k\tenO A)^Q$ is 
surjective, hence induces a bijection between the points of $Q$ on $A$ 
and on $k\tenO A$. Since $\alpha$ induces the identity on $k\tenO A$, it 
follows that $A$ fixes all points of $Q$ on $A$.
\end{proof}

\begin{Lemma} \label{foc1}
Let $\zeta : P\to$ $\CO^\times$ a group homomorphism such that 
$\foc(\CF)\leq$ $\ker(\zeta)$. Denote by $\eta$ the $\CO$-algebra 
automorphism of $\OP$ sending $u\in$ $P$ to $\zeta(u)u$. Let $Q$ be a 
subgroup of $P$, let $\varphi\in$ $\Hom_\CF(Q,P)$, and set $W=$ 
$\OP_\varphi \tenOQ \OP$. There is a unique isomorphism of 
$\OP$-$\OP$-bimodules
$W \cong {_\eta{W}_\eta}$
induced by the map sending a tensor $u\ten v$ to $\zeta(uv) u\ten v$,
where $u$, $v\in$ $P$.
\end{Lemma}

\begin{proof}
We need to show that the assignement $u\ten v\mapsto$ 
$\zeta(uv) u\ten v$ is well-defined. Let $u$, $v\in$ $P$ and $w\in$ $Q$. 
By the definition of $W$, the images of $u\varphi(w)\ten v$ and 
$u\ten wv$ in $W$ are equal. Thus we need to show that 
$u\varphi(w)\ten v$ and $u\ten wv$ have the same image under this
assignment. The image of $u\varphi(w)\ten v$ is 
$\zeta(u\varphi(w)v)u\varphi(w)\ten v=$
$\zeta(u\varphi(w)v)u\ten wv$. The image of $u\ten wv$ is 
$\zeta(uwv) u\ten wv$. Since $\zeta$ is a group 
homomorphism satisfying $\zeta(w)=$ $\zeta(\varphi(w))$, it follows 
that $\zeta(uwv)=$ $\zeta(u\varphi(w)v)$. This shows that the map
$u\ten v\mapsto$ $\zeta(uv)u\ten v$ is an $\CO$-linear isomorphism.
A trivial verification shows that this map is also a homomorphism of
$\OP$-$\OP$-bimdoules.
\end{proof}

\begin{Lemma} \label{foc2}
Let $\zeta : P\to$ $\CO^\times$ a group homomorphism such
that $\foc(\CF)\leq$ $\ker(\zeta)$. Denote by $\eta$ the
$\CO$-algebra automorphism of $\OP$ sending $u\in$ $P$ to $\zeta(u)u$.
We have an isomorphism of $\OP$-$\OP$-bimodules $A\cong$ 
${_\eta{A}_\eta}$ which induces the identity on $k\tenO A$.
\end{Lemma}

\begin{proof}
Note that $\eta$ induces the identity on $kP$ because the image
in $k$ of any $p$-power root of unity is $1$.
It follows from the main result in \cite{Puigloc} (see also 
\cite[Appendix]{Lisplendid} for an account of this material) that every 
indecomposable direct summand of $A$ as an $\OP$-$\OP$-bimodule is 
isomorphic to a bimodule of the form $\OP_\varphi \tenOQ \OP$ for 
some subgroup $Q$ of $P$ and some $\varphi\in$ $\Hom_\CF(Q,P)$. Thus 
\ref{foc2} follows from \ref{foc1}. 
\end{proof}

\begin{Lemma} \label{foc3}
Let $\zeta : P\to$ $\CO^\times$ a group homomorphism such that 
$\foc(\CF)\leq$ $\ker(\zeta)$. Denote by $\eta$ the $\CO$-algebra 
automorphism of $\OP$ sending $u\in$ $P$ to $\zeta(u)u$. Set $A^e=$ 
$A\tenO A^\op$. Set $\bar A=$ $k\tenO A$ and $\bar A^e=$ 
$\bar A\tenk \bar A^\op$. The canonical algebra homomorphism
$$\End_{A^e}(A_\eta\tenOP A)\to \End_{\bar A^e}(\bar A\tenkP\bar A)$$
is surjective.
\end{Lemma}

\begin{proof}
A standard adjunction yields a canonical linear isomorphism
$$\End_{A^e}(A\eta \tenOP A)\cong 
\Hom_{A \tenO \OP^\op}(A_\eta , A_\eta \tenOP A )\ .$$
Using $A_\eta \cong$ $A \tenOP \OP_\eta$, another
standard adjunction implies that this is isomorphic to
$$\Hom_{\OP\ten \OP^\op}(\OP_\eta , A_\eta \tenOP A )\ .$$
Using $\OP_\eta\cong$ ${_{\eta^{-1}}\OP}$ and `twisting' by
$\eta$ on the left side of both arguments, the previous expression is 
canonically isomorphic to
$$\Hom_{\OP\tenO \OP^\op}(\OP, {_\eta A}_\eta \tenOP A)\ .$$
Using \ref{foc2}, this is isomorphic to
$$\Hom_{\OP\tenO \OP^\op}(\OP, A\tenOP A)\ .$$
Since $A\tenOP A$ is a permutation $\CO(P\times P)$-module, it
follows that the canonical map
$$\Hom_{\OP\tenO \OP^\op}(\OP, A\tenOP A) \to
\Hom_{kP\tenk kP^\op}(kP, \bar A\tenkP \bar A)$$
is surjective. Since the previous isomorphisms commute
with the canonical surjections modulo $J(\CO)$, the result
follows.
\end{proof}

\begin{Lemma} \label{foc4}
Let $\zeta\in$ $\Hom(P,\CO^\times)$. 
Denote by $\eta$ the automorphism of $\OP$ defined by $\eta(u)=$ 
$\zeta(u)u$ for all $u\in$ $P$. Let $\alpha\in$ $\Aut_1(A)$. 
The following are equivalent.

\smallskip\noindent (i)
The class of $\alpha$ in $\Out(A)$ has a representative $\alpha'$
in $\Aut_1(A)$ which extends $\eta$.

\smallskip\noindent (ii)
There is an $A$-$\OP$-bimodule isomorphism $A_\eta\cong$ $A_\alpha$.

\smallskip\noindent (iii)
There is an $\OP$-$A$-bimodule isomorphism ${_\eta{A}}\cong$ 
${_\alpha{A}}$.

\smallskip\noindent (iv)
As an $A$-$A$-bimodule, $A_\alpha$ is isomorphic to a direct summand
of $A_\eta\tenOP A$.
\end{Lemma}

\begin{proof}
By \ref{autom-extend}, (i) implies (ii) and (iii). If (ii) holds,
then by \ref{autom-extend} there is $\alpha'\in$ $\Aut(A)$ which
extends $\eta$ and which represents the same class as $\alpha$ in
$\Out(A)$. In particular, $\alpha'\in$ $\Aut_1(A)\cdot\Inn(A)$.
It follows from \ref{zeta-points} that we may choose $\alpha'$ in
$\Aut_1(A)$. Thus (ii) implies (i). The analogous argument shows that
(iii) implies (i). Suppose again that (ii) holds. It follows from 
\ref{AalphaA} that $A_\alpha$ is isomorphic to a direct summand of 
$A_\alpha\tenOP A$, hence of $A_\eta\tenOP A$. Thus (ii) implies (iv). 
Note that $A_\alpha$ remains indecomposable as an $A$-$\OP$-bimodule. 
If (iv) holds, then in particular, $A_\alpha$ is isomorphic to a 
direct summand of $A_\eta\tenOP A$ as an $A$-$\OP$-bimodule, hence of 
$A_\eta\tenOP W$ for some indecomposable $\OP$-$\OP$-bimodule summand 
$W$ of $A$. As before, any such summand is of the form 
$OP_\varphi\tenOP \OP$ for some subgroup $Q$ of $P$ and some 
$\varphi\in$ $\Hom_\CF(Q,P)$. Since $A$, hence $A_\eta$, is projective 
as a right $\OP$-module, it follows that every indecomposable 
$A$-$\OP$-bimodule summand of $A_\eta\tenOP W$ has $\CO$-rank divisible 
by $|P|\cdot|P:Q|$. Now $|P|$ is the highest power of $p$ which divides 
the $\CO$-rank of $A$. Thus, as an $A$-$\OP$-bimodule, $A_\alpha$ is 
isomorphic to a direct summand of $A_\eta\tenOP \OP_\varphi$ for some 
$\varphi$ belonging to $\Aut_\CF(P)$. Any such $\varphi$ is induced by 
conjugation with some element in $A^\times$, and hence $A_\alpha$ is 
isomorphic to a direct summand of $A_\eta$, as an $A$-$\OP$-bimodule. 
Since both have the same $\CO$-rank, they are isomorphic. This shows 
that (iv) implies (ii) and concludes the proof.
\end{proof}

\begin{proof}[{Second proof of Lemma \ref{zeta-extend1}}]
Let $\zeta : P\to$ $\CO^\times$ be a group homomorphism such that 
$\foc(\CF)\leq$ $\ker(\zeta)$. Denote by $\eta$ the
$\CO$-algebra automorphism of $\OP$ sending $u\in$ $P$ to $\zeta(u)u$.
Set $\bar A=$ $k\tenO A$. The $\bar A$-$\bar A$-bimodule 
$\bar A\tenkP \bar A$ has a direct summand isomorphic to $\bar A$. 
It follows from standard lifting theorems of idempotents and
\ref{foc3} that the $A$-$A$-bimodule $A_\eta\tenOP A$
has an indecomposable direct summand $N$ satisfying $k\tenO N\cong$ 
$\bar A$. Then $N$ induces a Morita equivalence on $\mod(A)$ which
induces the identity on $\mod(k\tenO A)$. Thus $N\cong$ $A_\alpha$ 
for some $\alpha\in$ $\Aut_1(A)$. It follows from \ref{foc4}
that we may choose $\alpha$ in $\Aut_1(A)$ in such a way that 
$\alpha$ extends $\eta$. 
\end{proof}

\section{Proof of Theorem \ref{star-lift}}

We need the interpretation from \cite[\S 5]{Puigpoint} of the 
$*$-construction at the source algebra level. For $\chi$ a class 
function on $G$ and $u_\epsilon$ a pointed element on $\OG$ we set 
$\chi(u_\epsilon)=$ $\chi(uj)$ for some, and hence any, $j\in$ 
$\epsilon$. By \cite[4.4]{Puigpoint} the matrix of values 
$\chi(u_\epsilon)$, with $\chi\in$ $\Irr_K(B)$ and $u_\epsilon$ 
running over a set of representatives of the conjugacy classes of local 
pointed elements contained in $P_\gamma$ is the matrix of generalised 
decomposition numbers of $B$. This matrix is nondegenerate, and hence a 
character $\chi\in$ $\Irr_K(B)$ is determined by the values 
$\chi(u_\epsilon)$, with $u_\epsilon$ as before. By the description of 
the $*$-construction in \cite[\S 5]{Puigpoint}, for any local pointed 
element $u_\epsilon$ contained in $P_\gamma$ we have 
$$(\zeta*\chi)(u_\epsilon)= \zeta(u)\chi(u_\epsilon)\ .$$
The source algebra $A=$ $iBi$ and the block algebra $B$ are Morita 
equivalent via the bimodule $iB$ and its dual, which is isomorphic to 
$Bi$. Through this Morita equivalence, $\chi$ corresponds to an
irreducible character of $K\tenO A$, obtained from restricting $\chi$
from $B$ to $A$. Let $u_\epsilon$ be a local pointed element on $\OG$ 
contained in $P_\gamma$. Then there is $j\in$ $\epsilon$ such that $j=$ 
$ij=$ $ji$, hence such that $j\in$ $A$. In other words, the formula 
$(\zeta*\chi)(u_\epsilon)=$ $\zeta(u)\chi(u_\epsilon)$ describes indeed
the $*$-construction at the source algebra level.

By Lemma \ref{AextendB}, the above Morita equivalence between $A$ and 
$B$ induces a group isomorphism $\Out_1(A)\cong$ $\Out_1(B)$ obtained
from extending automorphisms of $A$ in $\Aut_1(A)$ to automorphisms
of $B$ in $\Aut_1(B)$. Composed with the group
homomorphism $\Psi$ from Theorem \ref{PfocAut}, this yields an injective
group homomorphism $\Phi : \Hom(P/\foc(\CF),\CO^\times)\to$ $\Out_1(B)$.
The uniqueness statement in Theorem \ref{star-lift} follows from the
uniqueness statement in Theorem \ref{PfocAut}.

In order to show that $\chi^\Phi(\zeta)=$ $\zeta*\chi$, it suffices
to prove that an automorphism $\alpha$ of $A$ representing the class 
$\Psi(\zeta)$ sends the character of $K\tenO A$ corresponding to $\chi$ 
to that corresponding to $\zeta*\chi$. By the above, and using the same
letter $\chi$ for the restriction of $\chi$ to $A$, it suffices to show 
that $\chi^{\alpha}(u_\epsilon)=$ $\zeta(u)\chi(u_\epsilon)$. By 
\ref{PfocAut}, we may choose $\alpha$ such that $\alpha(ui)=$ 
$\zeta(u)ui$. By \ref{zeta-points},  $\alpha$ fixes the point 
$\epsilon$; that is, $\alpha(j)\in$ $\epsilon$. It follows that 
$\chi^{\alpha}(u_\epsilon)=$ $\chi(\alpha(uj))=$ 
$\chi(\zeta(u)u\alpha(j))=$ $\zeta(u)\chi(u_\epsilon)$. This proves
the statement (i) of Theorem \ref{star-lift}. 

\begin{Remark}
The fact that the $*$-construction on characters lifts to the action
of automorphisms can be used to give an alternative proof of the fact 
that $\Phi$ (or equivalently, $\Psi$) is injective. If $\alpha$ as 
defined in the above proof is inner, then an automorphism of $B$ 
corresponding to $\alpha$ fixes any $\chi\in$ $\Irr_K(B)$, or 
equivalently, $\zeta*\chi=$ $\chi$ for any $\chi\in$ $\Irr_K(B)$. This 
however forces $\zeta=$ $1$ as the group  
$\Hom(P/\foc(\CF), \CO^\times)$ acts faithfully on $\Irr_K(B)$ via the 
$*$-construction; in fact, it acts freely on the subset of height zero 
characters in $\Irr_K(B)$ by \cite[\S 1]{Rob08}. 
\end{Remark}

For the proof of Theorem \ref{star-lift} (ii), assume that $\CO$ is
finitely generated over the $p$-adic integers (this assumption is
needed in order to quote results from \cite{HeKi} and \cite{Weiss}).
Let $\alpha\in$ $\Aut_m(A)$. Since $(1-\tau_p)\CO=$ $\pi^m\CO$, it 
follows that $\alpha$ induces the identity on $B/(1-\tau_p)B$. Thus 
$B_\alpha/(1-\tau_p)B_\alpha\cong$ $B/(1-\tau_p)B$. By Weiss' theorem 
as stated in \cite[Theorem 3.2]{HeKi}, $B_\alpha$ is a monomial 
$\CO(P\times P)$-module, hence so is $iB_\alpha$. Since $iB_\alpha$
is indecomposable as an $\CO(P\times G)$-module and relatively
$\CO(P\times P)$-projective, it follows that $iB_\alpha$ is a  linear 
source module. Since $\Delta P$ is a vertex of $k\tenO iB$, this is 
also a vertex of $iB_\alpha$. Thus there is $\zeta : P\to$ $\CO^\times$ 
such that $iB_\alpha$ is isomorphic to a direct summand of 
$\Ind^{P\times G}_{\Delta P}(\CO_\zeta)$, where $\CO_\zeta=$ $\CO$ with 
$(u,u)$ acting as multiplication by $\zeta(u)$ for all $u\in$ $P$. By 
\ref{zetaeta} we have $\Ind^{P\times P}_{\Delta P}(\CO_\zeta)\cong$ 
${_\eta\OP}$, where $\eta\in$ $\Aut_1(\OP)$ is defined by $\eta(u)=$ 
$\zeta(u)u$ for all $u\in$ $P$. Thus $iB_\alpha$ is isomorphic to a 
direct summand of ${_\eta B}$, hence isomorphic to ${_\eta jB}$ for some 
primitive idempotent $j\in$ $B^P$. The necessarily $j$ is a source
idempotent because $k\tenO iB_\alpha\cong$ $k\tenO iB$ and 
$\Br_P(i)\neq$ $0$. Since the local points of $P$ on $B$ are 
$N_G(P)$-conjugate, we may assume that $j$ is $N_G(P)$-conjugate to 
$i$, so after replacing both $j$ and $\zeta$ by an $N_G(P)$-conjugate 
we may assume that ${_\eta iB}\cong$ $iB_\alpha$. We also may assume 
that $\alpha$ fixes $i$. Then multiplication on the right by $i$ 
implies that ${_\eta A}\cong$ $A_\alpha$, where we abusively use the 
same letter $\alpha$ for the automorphism of $A$ obtained from 
restricting the automorphism $\alpha$ on $B$. Since $A_\alpha\cong$ 
${_{\alpha^{-1}}A}$, it follows from \ref{foc4} that $\alpha$ can be 
chosen to extend $\eta^{-1}$. It follows from \ref{zeta-extend2} that 
$\foc(\CF)\leq$ $\ker(\zeta)$. This shows that the class of $\alpha$ is 
equal to $\Psi(\zeta^{-1})$. It remains to show that $\zeta$ has values 
in the subgroup $\mu$ of order $p$ of $\CO^\times$. Since $\Phi$ is
injective, it suffices to show that the class of $\alpha$ has order
at most $p$ in $\Out(A)$, or equivalently, that $\alpha^p$ is inner.
Since $\alpha\in$ $\Aut_m(A)$, an easy calculation shows  
that $\alpha^p\in$ $\Aut_{m+1}(A)$. It follows from \cite[3.13]{HeKi}
that $\Out_{m+1}(A)$ is trivial, thus $\alpha^p$ is indeed inner. 
This concludes the proof of Theorem \ref{star-lift}.  

\medskip\noindent
{\it Acknowledgement.} The present work is partially funded by
the EPSRC grant EP/M02525X/1. 


\end{document}